\documentclass[12pt,reqno]{amsart}
\usepackage{amssymb,latexsym}
\usepackage{amsthm}
\usepackage{mathtools}
\usepackage{url}

\makeatletter
\@namedef{subjclassname@2010}{%
  \textup{2010} Mathematics Subject Classification}
\makeatother

\newtheorem{thm}{Theorem}[section]
\newtheorem{cor}[thm]{Corollary}
\newtheorem{lem}[thm]{Lemma}
\newtheorem{prop}[thm]{Proposition}

\theoremstyle{definition}
\newtheorem{defin}[thm]{Definition}

\newtheorem{exa}[thm]{Example}

\numberwithin{equation}{section}

\newcommand{\Z}{\mathbb Z}
\newcommand{\Q}{\mathbb Q}
\newcommand{\R}{\mathbb R}
\newcommand{\op}{\operatorname}

\begin{document}

\title[Class numbers of totally real fields and the Weber problem]{Class numbers of totally real fields and applications to the Weber class number problem}

\author[John C. Miller]{John C. Miller}
\address{Department of Mathematics\\ Rutgers University\\
Hill Center for the Mathematical Sciences\\110 Frelinghuysen Road\\Piscataway, NJ 08854-8019}
\email{jcmiller@math.rutgers.edu}

\date{}

\begin{abstract}
The determination of the class number of totally real fields of large discriminant is known to be a difficult problem.  The Minkowski bound is too large to be useful, and the root discriminant of the field can be too large to be treated by Odlyzko's discriminant bounds.  We describe a new technique for determining the class number of such fields, allowing us to attack the class number problem for a large class of number fields not treatable by previously known methods.  We give an application to Weber's class number problem, which is the conjecture that all real cyclotomic fields of power of 2 conductor have class number 1.
\end{abstract}



\maketitle

\section{Introduction}
Although the class number is a fundamental invariant of number fields, the problem of determining the class number is rather difficult for fields of large discriminant.  Even cyclotomic fields of relatively small conductor have discriminants too large for their class numbers to be calculated.

The difficulty is that the Minkowski bound for a totally real field of degree $n$,
\[M(K) = \frac{n!}{n^n} \sqrt{|d(K)|},\]
is often far too large to be useful.  For example, to prove that the real cyclotomic field of conductor 256 has class number 1 using the Minkowski bound, we would need to check that every prime integer below the Minkowski bound factors into principal prime ideals, requiring us to check more than $10^{78}$ primes!

The approach of using Odlyzko's discriminant bounds can handle fields of larger discriminant than using the Minkowski bound, but this technique, as applied by Masley \cite{Masley} and van der Linden \cite{Linden}, encountered a barrier:  Odlyzko's discriminant lower bounds could only establish an upper bound for the class number of a totally real field of degree $n$ if its \emph{root discriminant}, the $n$th root of the discriminant, was sufficiently small.

We will overcome this barrier by establishing lower bounds for sums over the prime ideals of the Hilbert class field.

\section{Application to the Weber class number problem}
Weber \cite{Weber} studied the class numbers of the real cyclotomic fields $\Q(\zeta_{2^k}+\zeta_{2^k}^{-1})$, and proved that their class numbers are odd for all $k$.  Fukuda and Komatsu \cite{Fukuda} went much further and proved that no primes less than $10^9$ can divide these class numbers, which suggests that the class numbers of these fields may in fact all be $1$.  This conjecture, known as \emph{Weber's class number problem}, is also supported by Cohen-Lenstra heuristics \cite{Buhler}. 

Using Odlyzko's discriminant bounds, van der Linden \cite{Linden} proved that the class number of $\Q(\zeta_{128}+\zeta_{128}^{-1})$ is $1$, and, under the assumption of the generalized Riemann hypothesis (GRH), proved that $\Q(\zeta_{256}+\zeta_{256}^{-1})$ has class number $1$.  However, due to their rather large discriminants, his method could neither unconditionally prove that $\Q(\zeta_{256}+\zeta_{256}^{-1})$ has class number $1$, nor could it be applied to $\Q(\zeta_{512}+\zeta_{512}^{-1})$, even under the assumption of the generalized Riemann hypothesis.

However, by counting sufficiently many prime ideals of the Hilbert class field, we overcome the problem of the large discriminants, and prove the following.

\begin{thm}\label{MainResult256}
The class number of the real cyclotomic field $\Q(\zeta_{256}+\zeta_{256}^{-1})$ is $1$.
\end{thm}

\begin{thm}\label{MainResult512}
Under the assumption of the generalized Riemann hypothesis, the class number of the real cyclotomic field $\Q(\zeta_{512}+\zeta_{512}^{-1})$ is $1$.
\end{thm}

The above theorems, together with knowledge of the relative class number \cite[pg. 412]{Washington}, allow us to determine the class number of the full cyclotomic field.

\begin{cor}\label{CorResult256}
The class number of the cyclotomic field $\Q(\zeta_{256})$ is \[10{,}449{,}592{,}865{,}393{,}414{,}737.\]
\end{cor}

\begin{cor}\label{CorResult512}
Under the generalized Riemann hypothesis, the class number of the cyclotomic field $\Q(\zeta_{512})$ is \[6{,}262{,}503{,}984{,}490{,}932{,}358{,}745{,}721{,}482{,}528{,}922{,}841{,}978{,}219{,}389{,}975{,}605{,}329.\]
\end{cor}

The latter field has degree $256$ and a discriminant of approximately $3 \times 10^{616}$.  The author is unaware of any other number field of such large degree or discriminant for which the class number has been calculated under the generalized Riemann hypothesis.

It is striking that the class number of the real subfield is so small compared to the class number of the full cyclotomic field.

\section{Upper bounds for class numbers of fields of small discriminant}
In this section we briefly review the theory of root discriminants and the application of Odlyzko's discriminant bounds to find upper bounds for class numbers of totally real fields.  Further details can be found in \cite{Linden}, \cite{Masley} and \cite{Odlyzko}.

\begin{defin}
Let $K$ denote a number field of degree $n$ over $\mathbb Q$.  Let $d(K)$ denote its discriminant.  The \emph{root discriminant} $\operatorname{rd}(K)$ of $K$ is defined to be:
\[\operatorname{rd}(K)=|d(K)|^{1/n}.\]
\end{defin}

\begin{prop}
Let $L/K$ be an extension of number fields.  Then
\[\operatorname{rd}(K) \leq \operatorname{rd}(L),\]
with equality if and only if $L/K$ is an unramified at all finite primes, i.e. no prime ideal of $K$ ramifies in $L$.
\end{prop}

\begin{proof}
The discriminants of $K$ and $L$ are related by the formula
\[|d(L)| = N(d(L/K))|d(K)|^{[L:K]}\]
where $d(L/K)$ denotes the relative discriminant ideal and $N$ denotes the absolute norm of the ideal, from which the first statement follows.  A prime of $K$ ramifies in $L$ if and only if the prime divides the relative discriminant $d(L/K)$.  Thus, $L/K$ is unramified if and only if $d(L/K)$ is the unit ideal, proving the second statement.
\end{proof}

\begin{cor}
Let $K$ be a number field.  Then the Hilbert class field of $K$ has the same root discriminant as $K$.
\end{cor}

Odlyzko constructed a table \cite{OdlyzkoTable} of pairs $(A,E)$ such that a totally real field $K$ of degree $n$ has a lower bound for its discriminant,
\[|d(K)| > A^n e^{-E}.\]
If we apply Odlyzko's discriminant bounds to the Hilbert class field of $K$ and use the above corollary, we get
\[ \log{\op{rd}(K)} > \log{A} - \frac{E}{hn}.\]
If $\op{rd}(K) < A$, then we obtain an upper bound for the class number $h$,
\begin{equation} \label{eqref:ClassNumberBound}
h < \frac{E}{n(\log{A} - \log{\op{rd}(K))}}.
\end{equation}
Thus, we can establish a class number upper bound for fields of small root discriminant.

However, this method, used by Masley and van der Linden, encounters an obstacle if the field has large discriminant.  If the root discriminant is larger than the maximum $A$ in Odlyzko's table, the above method can not establish a class number upper bound.  The maximum $A$ in Odlyzko's table is $60.704$ (or $213.626$ under the assumption of the generalized Riemann hypothesis).
\section{An identity for the class number of a totally real field}
Let $F$ be an Schwarz class function on $\R$ with $F(0)=1$ and $F(-x)=F(x)$.  Let $\Phi$ be defined by
\[\Phi(s) = \int_{-\infty}^{\infty} F(x) e^{(s - 1/2)x} \, dx.\]

Let $K$ be a  number field of degree $n$ with $r_1$ real embeddings.  Poitou's version \cite{Poitou} of Weil's ``explicit formula" for the Dedekind zeta function of $K$ states that
\begin{align*}
\log d(K) &= r_1\frac{\pi}{2} + n(\gamma + \log 8\pi) - n \int_0^{\infty} \frac{1-F(x)}{2 \sinh \frac{x}{2}} \, dx \\ & \quad - r_1 \int_0^{\infty} \frac{1-F(x)}{2 \cosh \frac{x}{2}} \, dx - 4 \int_0^{\infty} F(x) \cosh \frac{x}{2} \, dx \\ & \quad + \sum_{\rho} \Phi(\rho) + 2 \sum_{\mathfrak P} \sum_{m=1}^{\infty} \frac{\log N\mathfrak P}{N\mathfrak P^{m/2}}F(m \log N\mathfrak P)
\end{align*}
where $\gamma$ is Euler's constant.  The first sum is over the nontrivial zeros of the Dedekind zeta function of $K$, and the second sum is over the prime ideals of $K$.

Let $K$ be a totally real field.  We can apply the explicit formula to the Hilbert class field $H(K)$ of $K$.  Let $h$ denote the class number of $K$.  Since
\[\log d(H(K)) = hn \log \op{rd}(H(K)) = hn \log \op{rd}(K),\]
we have
\begin{align*}
hn \log \op{rd}(K) &= hn\left(\frac{\pi}{2} + \gamma + \log 8\pi - \int_0^\infty \frac{1-F(x)}{2}\left( \frac{1}{\sinh{\frac{x}{2}}} + \frac{1}{\cosh{\frac{x}{2}}} \right) dx \right)\\ & \quad - 4 \int_0^{\infty} F(x) \cosh \frac{x}{2} \, dx + \sum_{\rho} \Phi(\rho) \\ & \quad + 2 \sum_{\mathfrak P} \sum_{m=1}^{\infty} \frac{\log N\mathfrak P}{N\mathfrak P^{m/2}}F(m \log N\mathfrak P)
\end{align*}
where the two sums are now over the nontrivial zeros of the Dedekind zeta function of $H(K)$ and the prime ideals of $H(K)$ respectively.  We rearrange this to get the identity

\begin{equation} \label{eqref:ClassNumberIdentity}
h = \frac{4 \mathcal H(F) / n}{C - \mathcal{G}(F) - \log \operatorname{rd}(K) + \frac{1}{hn} \sum_{\rho} \Phi(\rho) + \frac{2}{hn} \sum_{\mathfrak P} \sum_{m=1}^{\infty} \frac{\log N\mathfrak P}{N\mathfrak P^{m/2}}F(m \log N\mathfrak P)}
\end{equation}
where
\[C = \frac{\pi}{2} + \gamma + \log{8\pi},\]
\[\mathcal{G}(F) = \int_0^\infty \frac{1-F(x)}{2}\left( \frac{1}{\sinh{\frac{x}{2}}} + \frac{1}{\cosh{\frac{x}{2}}} \right) \, dx,\]
and
\[\mathcal{H}(F) = \int_0^\infty F(x) \cosh{\frac{x}{2}} \, dx.\]

Suppose we choose $F$ so that $F$ is nonnegative and so that $\Phi(\rho) \geq 0$ for all nontrivial zeros $\rho$.  If it is true that
\[C - \mathcal{G}(F) - \log \operatorname{rd}(K) > 0\]
then we get the upper bound for the class number,
\[h \leq \frac{4 \mathcal H(F)}{n}\frac{1}{C - \mathcal{G}(F) - \log \operatorname{rd}(K)}.\]
If we choose $F$ appropriately, we can recover the class number upper bounds (\ref{eqref:ClassNumberBound}) obtained by Odlyzko's discriminant bounds.

\section{Upper bounds for class numbers of fields of large discriminant}
If a number field has root discriminant greater than $4\pi e^{\gamma + 1} = 60.839...$ (or greater than $8\pi e^{\gamma + \pi/2} = 215.33...$ if GRH is assumed), then we will have
\[C - \mathcal{G}(F) - \log \operatorname{rd}(K) < 0,\]
so we can not establish a class number upper bound following the approach of the previous section.

However, \emph{if we have further knowledge of the zeros or the prime ideals of the Hilbert class field}, then we may be able establish a nontrivial lower bound for the sums
\[\frac{1}{hn} \sum_{\rho} \Phi(\rho)\]
or
\[ \frac{2}{hn} \sum_{\mathfrak P} \sum_{m=1}^{\infty} \frac{\log N\mathfrak P}{N\mathfrak P^{m/2}}F(m \log N\mathfrak P)\]
that is sufficiently large as to ensure a positive lower bound for the denominator of (\ref{eqref:ClassNumberIdentity}). Thus we may obtain a class number upper bound for fields with discriminants too large to have been treated by earlier methods.  Since it is difficult to make any explicit estimates of the low-lying zeros of the zeta function of the Hilbert class field, we must rely on the contribution of the prime ideals.

In his proof \cite{Linden} that $\Q(\zeta_{128}+\zeta_{128}^{-1})$ has class number 1, van der Linden used the contribution from the ramified prime above 2 to establish a class number upper bound.  Unfortunately, there is only one ramified prime and its contribution in not sufficient for conductors 256 or 512.  Fortunately, we can use the many unramified primes.

Suppose a prime integer $p$ totally splits in the field $K$ into principal prime ideals.  Since principal ideals totally split in the Hilbert class field, we have $hn$ prime ideals in the Hilbert class field that lie over $p$, each with a norm of $p$.  Thus for such $p$ we get a contribution to the prime ideal term of the explicit formula
\[\frac{2}{hn} \sum_{\mathfrak P|p} \sum_{m=1}^\infty \frac{\log N\mathfrak P}{N\mathfrak P^{m/2}}F(m \log N\mathfrak P) = 2 \sum_{m=1}^\infty \frac{\log p}{p^{m/2}}F(m \log p).\]

The assumption of the generalized Riemann hypothesis now takes on critical importance.  Without assuming that the nontrivial zeros lie on the critical line, the function $F$ would have to be chosen so that $\Phi$ is nonnegative on the entire critical strip.  Thus $F$ would have to be of the form
\[F(x) = \frac{f(x)}{\cosh \frac{x}{2}},\]
with $f$ nonnegative and a nonnegative Fourier transform \cite{Odlyzko}.  Such a condition would force $F$ to decay so rapidly that many prime ideals may be needed contribute significantly to the explicit formula.

If, on the other hand, we assume truth of the generalized Riemann hypothesis, specifically that the nontrivial zeros of the zeta function of the Hilbert class field lie on the critical line $\frac{1}{2}+it$, then $F$ would only required to be nonnegative with nonnegative Fourier transform.  For example,  $F$ could be chosen to be the Gaussian function,
\[F(x) = e^{-(x/c)^2}\]
for some positive constant $c$.  For large $c$, this decays less rapidly, allowing for a larger contribution from the prime ideals.

We summarize the above discussion with the following two lemmas.  If we do not assume the generalized Riemann hypothesis,
\begin{lem}\label{SplitPrimeLemmaNoRH}
Let $K$ be a totally real field of degree $n$, and let
\[F(x) = \frac{e^{-\left(x/c\right)^2}}{\cosh{\frac{x}{2}}}\]
for some positive constant $c$.  Suppose $S$ is a subset of the prime integers which totally split into principal prime ideals of $K$.  Let
\begin{align*}
B &= \frac{\pi}{2} + \gamma + \log{8\pi}  - \log \operatorname{rd}(K) - \int_0^\infty \frac{1-F(x)}{2}\left( \frac{1}{\sinh{\frac{x}{2}}} + \frac{1}{\cosh{\frac{x}{2}}} \right) \, dx \\ & \quad + 2 \sum_{p \in S} \sum_{m=1}^\infty \frac{\log p}{p^{m/2}}F(m \log p).
\end{align*}
If $B > 0$ then we have an upper bound for the class number $h$ of $K$,
\[h < \frac{2c \sqrt{\pi}}{nB}.\]
\end{lem}

On the other hand, if we do assume the truth of the generalized Riemann hypothesis, we have the following lemma.
\begin{lem}\label{SplitPrimeLemma}
Let $K$ be a totally real field of degree $n$, and let
\[F(x) = e^{-(x/c)^2}\]
for some positive constant $c$.  Suppose $S$ is a subset of the prime integers which totally split into principal prime ideals of $K$.  Let
\begin{align*}
B &= \frac{\pi}{2} + \gamma + \log{8\pi}  - \log \operatorname{rd}(K) - \int_0^\infty \frac{1-F(x)}{2}\left( \frac{1}{\sinh{\frac{x}{2}}} + \frac{1}{\cosh{\frac{x}{2}}} \right) \, dx \\ & \quad + 2 \sum_{p \in S} \sum_{m=1}^\infty \frac{\log p}{p^{m/2}}F(m \log p).
\end{align*}
If $B > 0$ then we have, under the generalized Riemann hypothesis, an upper bound for the class number $h$ of $K$,
\[h < \frac{2c \sqrt{\pi} e^{(c/4)^2}}{nB}.\]
\end{lem}

\section{The class number of $\Q(\zeta_{256}+\zeta_{256}^{-1})$}
Let $\Q(\zeta_m+\zeta_m^{-1})$ denote the $m$th real cyclotomic field, i.e. the maximal real subfield of the cyclotomic field $\Q(\zeta_m)$, where $\zeta_m$ is a primitive $m$th root of unity.  If $m$ is a power of $2$, then the discriminant of the real cyclotomic field is given by
\[d(\Q(\zeta_{2^k}+\zeta_{2^k}^{-1})) = 2^{(k-1)2^{k-2}-1}\]
and the root discriminant is
\[\operatorname{rd}(\Q(\zeta_{2^k}+\zeta_{2^k}^{-1})) = d(\Q(\zeta_{2^k}+\zeta_{2^k}^{-1}))^{2^{-(k-2)}} = 2^{k - 1 - 2^{-k+2}}.\]
In particular, we have
\[\operatorname{rd}(\Q(\zeta_{256}+\zeta_{256}^{-1})) =127.9891....\]
This root discriminant is too large to use Odlyzko's unconditional discriminant bound tables to establish an upper bound for the class number.  Therefore, we must show a sufficiently large contribution by prime ideals of small norm to the explicit formula in order to get an upper bound for the class number.  Although it is not difficult to find principal prime ideals of small norm in the $\Q(\zeta_{256}+\zeta_{256}^{-1})$, an unconditional proof that  $\Q(\zeta_{256}+\zeta_{256}^{-1})$ has class number $1$ will require that we exhibit principal prime ideals for a rather large number of primes.

We define the \emph{norm} of an element $x$ of a Galois number field $K$ to be
\[N(x) = \left| \prod_{\sigma \in \op{Gal}(K/\Q)} \sigma(x) \right|.\]

A prime integer $p$ totally splits in $\Q(\zeta_{256}+\zeta_{256}^{-1})$ if and only if $p$ is congruent to $\pm 1$ modulo $256$.  Thus, if there exists an algebraic integer with norm $p$ congruent to $\pm 1$ modulo $256$, then $p$ totally splits into principal prime ideals.

Let $\mathcal{O}$ denote the ring of integers of  $\Q(\zeta_{256}+\zeta_{256}^{-1})$.  Let $b_0 = 1$ and $b_j = 2 \cos{\frac{2 \pi j}{256}}$ for $j=1,\dots,63$, and let
\[c_k = \sum_{j=0}^k b_j.\]
for $k=0,\dots,63$.  Then $c_0,\dots,c_{63}$ is the basis for $\mathcal{O}$ that we will use.

Our strategy will be to search over a large number of ``sparse" vectors with respect to the basis $c_o,\dots,c_{63}$, i.e. vectors where almost all the coefficients are zero.  We make a list of those elements of $\mathcal{O}$ that have norms which are prime and are congruent to $\pm 1$ modulo $256$.  We will need tens of thousands of these primes to successfully establish an unconditional upper bound for the class number.

We consider all $x \in \mathcal{O}$ of the form
\[x = c_0 + a_1 c_{j_1} + a_2 c_{j_2} + a_3 c_{j_3} + a_4 c_{j_4} +a_5 c_{j_5} +a_6 c_{j_6},\]
with $1 \leq j_1 < j_2 < j_3 < j_4 < j_5 < j_6  \leq 63$ and $a_i \in  \{-1, 0, 1\}$ for $i=1,\dots, 6$.   Let $T$ be the set of all such $x$ .

The ideal $2\mathcal{O}$ is totally ramified.  Thus, if $x \in \mathcal{O}$ has even norm $N(x)$, we can divide $x$ by any element of norm $2$, say $b_1$, to get an algebraic integer $b_1^{-1} x$ with norm $N(x)/2$.  Therefore we consider the odd parts of all norms $N(x)$.  We define the set $U$ to be
\[U = \{\text{odd part of } N(x) | x \in T \}.\]

Let $S_1$ be the set of primes
\[S_1 = \{m : m \in U, m \text{ prime, }  m \equiv \pm 1 \,(\bmod \,256), m < 10^9 \}.\]
The set $S_1$ does not contain enough primes to establish a class number upper bound.  To supplement these primes, we can factor composites in $U$ using primes from $S_1$.  Let $S_2$ be set of primes
\[S_2 = \{p : pq \in U, p \text{ prime, } p \notin S_1, q \in S_1  \},\]
noting that if $N(x) = pq$ and $N(y)=q$, for $x, y \in \mathcal{O}$ for distinct primes $p$ and $q$, then $x/\sigma(y)$ is in $\mathcal{O}$ with norm $p$ for some Galois automorphism $\sigma$.

Let $S = S_1 \cup S_2$.  We apply Lemma \ref{SplitPrimeLemmaNoRH}, choosing $c = 210$ and putting
\[F(x) = \frac{e^{-\left(x/c\right)^2}}{\cosh{\frac{x}{2}}}.\]
A lower bound for the contribution from split primes is
\[2 \sum_{p \in S} \sum_{m=1}^\infty \frac{\log p}{p^{m/2}}F(m \log p) > 2 \sum_{p \in S} \sum_{m=1}^2 \frac{\log p}{p^{m/2}}F(m \log p) > 0.7023.\]
This still is not quite enough, so we supplement our prime ideal contribution by considering the totally ramified prime $2$.  This factors as $2\mathcal{O} = P^{64}$ where $P$ is a principal prime ideal of norm $2$, giving a contribution
\[\frac{2}{64} \sum_{m=1}^\infty \frac{\log p}{p^{m/2}}F(m \log p) > \frac{2}{64} \sum_{m=1}^{20} \frac{\log p}{p^{m/2}}F(m \log p) > 0.0331.\]

We have that $\log{\op{rd}( \Q(\zeta_{256}+\zeta_{256}^{-1})} = 4.8412...$ and we use numerical integration to find that
\[\mathcal{G}(F) = \int_0^\infty \frac{1-F(x)}{2}\left( \frac{1}{\sinh{\frac{x}{2}}} + \frac{1}{\cosh{\frac{x}{2}}} \right) \, dx < 1.2642.\]
Then we have
\begin{align*}
B &= \left(\frac{\pi}{2} + \gamma + \log{8\pi}\right)  - \log \operatorname{rd}(K) - \mathcal{G}(F) + 2 \sum_{\mathfrak P} \sum_{m=1}^{\infty} \frac{\log N\mathfrak P}{N\mathfrak P^{m/2}}F(m \log N\mathfrak P) \\ &> 5.3721 - 4.8412 - 1.2642 + 0.7023 + 0.0331 = 0.0021.
\end{align*}
Thus, we get a class number upper bound of
\[h < \frac{2c \sqrt{\pi}}{nB} < 5539.\]
Finally, we apply the results of Fukuda and Komatsu \cite{Fukuda} to prove that $h = 1$, proving Theorem \ref{MainResult256}.

\section{The class number of $\Q(\zeta_{512}+\zeta_{512}^{-1})$}
The field $\Q(\zeta_{512}+\zeta_{512}^{-1})$ has a root discriminant
\[\operatorname{rd}(\Q(\zeta_{512}+\zeta_{512}^{-1})) =254.6175....\]
which exceeds $8\pi e^{\gamma + \pi/2} = 215.33...$.  Therefore, we must show a contribution to the explicit formula by prime ideals of small norm to get an upper bound for the class number, even under the assumption of the generalized Riemann hypothesis.

In contrast to the unconditional proof that $\Q(\zeta_{256}+\zeta_{256}^{-1})$ has class number $1$, under the generalized Riemann hypothesis proving that $\Q(\zeta_{512}+\zeta_{512}^{-1})$ has class number $1$ will take knowledge of just a few principal prime ideals of small norm.  However, generators of those ideals will be rather difficult to find.

In the next section we will prove the following lemma.
\begin{lem}\label{PrimeLemma}
In the real cyclotomic field $\Q(\zeta_{512}+\zeta_{512}^{-1})$, there exist algebraic integers of norms $3583$, $5119$, $6143$, $7681$, $8191$, $10753$, $11777$, $12289$, $12799$ and $13313$.
\end{lem}

The existence of these algebraic integers of small prime norm allows us to prove the main result.

\begin{proof}[Proof of Theorem \ref{MainResult512}]
Let $K= \Q(\zeta_{512}+\zeta_{512}^{-1})$, which has degree 128 over $\Q$.  Let $h$ be its class number.  Let $F$ be the function
\[F(x) = e^{-(x/c)^2},\]
with $c = 8.7$.

The following integral can be calculated using numerical integration:
\[\int_0^\infty \frac{1-F(x)}{2}\left( \frac{1}{\sinh{\frac{x}{2}}} + \frac{1}{\cosh{\frac{x}{2}}} \right) \, dx < 0.3358.\]

A prime integer $p$ totally splits in $K$ if and only if $p$ is congruent to $\pm 1$ modulo $512$.  Let $S$ denote the set of primes,
\[\{3583, 5119, 6143, 7681, 8191, 10753, 11777, 12289, 12799, 13313\},\]
which are the ten smallest prime integers which are congruent to $\pm 1$ modulo $512$.  Using Lemma \ref{PrimeLemma}, there is a lower bound for the contribution from the prime ideals,
\[ 2 \sum_{p \in S} \sum_{m=1}^\infty \frac{\log p}{p^{m/2}}F(m \log p) > 2 \sum_{p \in S} \frac{\log p}{\sqrt{p}}F(\log p) > 0.6898.\]

We apply Lemma \ref{SplitPrimeLemma} to get an upper bound for the class number,
\[h < 147.\]

Finally, we can apply the ``rank corollary" from \cite{Masley} (or the results of \cite{Fukuda}) to see that $h=1$.

\end{proof}

\section{Algebraic integers of small prime norm in $\Q(\zeta_{512}+\zeta_{512}^{-1})$}

\begin{proof}[Proof of Lemma \ref{PrimeLemma}]
It suffices to explicitly provide the elements which have the desired norms.  The real cyclotomic field $\Q(\zeta_{512}+\zeta_{512}^{-1})$ has an integral basis, $\{b_0, b_1,\dots, b_{127}\}$ where $b_0 = 1$ and $b_j = 2 \cos{(2 \pi j/512)}$ for $j$ from $1$ to $127$.  Given an element $(a_j)$ of the field in this basis, the norm of $(a_j)$ is the absolute value of
\[\prod_{k=0}^{127} \left (a_0 + \sum_{j=1}^{127} a_j \cos{\frac{2 \pi j(2k+1)}{512}} \right )\]
Using this basis, we list ten algebraic integers.  

This element has norm $3583$:

\tiny\noindent
[549, 471, 40, 400, 546, 13, 144, $-222$, 769, 1114, 4, 109, 1498, $-48$, $-272$, 1393, 337, 295, $-304$, 262, 653, $-5$, 487, 991, 1080, 604, $-176$, 147, 517, 299, $-136$, 5, 331, 1051, 943, 158, 281, 9, $-299$, $-337$, 685, 105, 65, 981, 1039, $-104$, $-316$, 999, 519, 195, 361, 367, 1033, 556, 435, 533, 126, $-393$, 391, 1413, 100, 142, 373, 268, $-875$, $-246$, $-117$, $-327$, 1530, 695, $-210$, 1137, 844, $-882$, 101, 254, $-347$, 281, 441, 1727, 909, $-729$, $-397$, $-117$, 478, $-947$, 67, 1040, 445, 138, 154, 473, 412, 324, $-164$, 625, $-50$, 156, 141, $-7$, 376, $-985$, $-434$, 1002, 503, $-343$, $-204$, $-200$, $-67$, 170, $-922$, 554, 867, $-172$, 29, 387, 797, $-470$, 155, 42, $-270$, $-14$, 31, 246, 385, 162, $-137$, 197].
\normalsize

This element has norm $5119$:

\tiny\noindent
[147, $-104494$, $-26676$, 12081, 25706, $-14209$, 71256, 99827, 36209, $-47677$, 66855, 65451, 4681, $-88975$, $-15784$, 32245, 41017, 45678, $-5821$, 127438, 17275, 161270, 121388, 141018, 76565, 18507, $-25523$, 8820, 86486, $-3883$, $-59945$, $-32692$, 7427, 168170, 79532, 111518, $-40813$, $-721$, 100225, 38681, 35033, $-59976$, $-26151$, $-150361$, 17703, $-10107$, 7624, $-39793$, $-74576$, 12244, 49328, 108034, 79004, $-83833$, $-31377$, 723, 70856, $-19714$, 6073, 22609, 4054, 29678, 26444, 144109, $-16167$, 13697, 4492, 36832, 68459, 100913, 66179, 7047, $-3034$, 156125, 61044, 32403, $-6778$, 114846, $-30960$, 2675, 25809, $-21964$, 1166, $-119242$, 24160, 13870, 29732, 10150, 24991, 54782, 55211, 12440, $-65770$, $-63049$, $-36834$, $-77524$, 18444, $-165290$, 500, $-59284$, 36279, 53748, 34020, 9670, 13433, 81430, 31887, 115248, $-13390$, $-87277$, $-73639$, $-784$, 62328, $-25731$, $-8249$, 68768, 9913, 136174, 153369, 108430, $-60208$, 10978, $-25491$, 27206, 4128, $-8680$, $-41807$, $-88057$].
\normalsize

This element has norm $6143$:

\tiny\noindent
[687, 1109, $-264$, $-409$, 1118, 826, $-717$, 215, $-14$, 920, 22, $-20$, 1564, $-1030$, 424, 959, 90, $-540$, 374, 435, $-334$, 207, 140, 65, 841, $-339$, 124, 378, $-376$, 114, $-760$, 672, 232, $-973$, 341, $-71$, $-284$, 495, 329, $-106$, $-246$, 78, 301, $-475$, $-756$, 1359, $-410$, 441, 265, $-392$, 1402, $-27$, 320, 599, 1365, 258, $-473$, 416, 260, 1033, 197, 212, 1541, $-1026$, 688, 1377, $-1154$, 743, 406, 298, 127, $-1017$, 7, $-8$, 987, 440, $-730$, 199, 359, $-1041$, $-664$, 706, $-612$, $-125$, 1, 104, $-702$, $-215$, 335, 4, 725, $-88$, $-497$, $-665$, $-557$, 590, $-346$, 856, 338, $-862$, 1369, $-709$, 40, 303, 711, 783, $-572$, 282, 68, $-528$, 837, 882, 565, 165, $-50$, 41, $-535$, 299, $-351$, 1012, 15, $-183$, $-18$, $-615$, 758, $-158$, $-234$, 1738].
\normalsize

This element has norm $7681$:

\tiny\noindent
[12419, $-72$, 3815, 1193, $-3972$, $-4639$, $-9741$, $-525$, 20798, $-2284$, $-3016$, 2627, 13769, 7618, 9084, 5902, 7104, 2023, 7378, 576, 16966, 1470, $-15719$, 1047, 4681, 1683, 9320, $-2609$, 6279, 3161, 1227, 4325, 5423, $-2032$, 1901, 4788, 15042, $-4879$, 2991, 4479, 11213, 11266, 1431, $-17$, 16203, 4789, 7726, $-3520$, $-5160$, $-2409$, $-8557$, 5297, 9307, $-4523$, 3415, $-4331$, $-221$, 4670, $-1272$, $-5870$, 532, 637, 1065, $-415$, 14452, $-7845$, 3158, 12392, 3004, 5689, 5914, $-4077$, 18668, 9144, $-4237$, $-8474$, $-7417$, 1399, 1158, 4120, 2845, $-6194$, 3372, 5412, 5860, 5527, $-3739$, $-389$, 3600, 32, 5343, $-4709$, 8512, 2466, 902, 11131, 8876, 187, 11267, 104, 6654, 2458, $-10487$, $-2800$, $-10$, 1120, $-5029$, $-5069$, $-5301$, $-5294$, 7063, 5281, $-10073$, 7, $-5930$, 2933, 16618, $-8242$, $-4914$, $-2260$, $-1814$, 4186, 5099, 2296, 4516, $-1349$, 10952, 1756].
\normalsize

This element has norm $8191$:

\tiny\noindent
\begin{flushleft}
[12200476407, $-292487755$, 12237320977, $-308254192$, 12300523691, $-486916755$, 12065092025, $-758732357$, 12260434117, $-988892432$, 12200465663, $-959314971$, 12023444230, $-807694380$, 11534015416, $-285959572$, 11389976464, 263767191, 11494483606, 118977394, 11822670966, $-186682783$, 11780632874, $-30299156$, 11717352645, 106809269, 11245815188, 100697928, 11092639613, $-233023507$, 11092354501, $-231762119$, 10881098713, 220511907, 9986637594, 698898280, 9678586036, 759689527, 9976530302, 456120687, 10566559282, 80998844, 10565421255, 498544865, 10139442513, 1100466640, 9626488800, 1042764202, 9674228265, 454041510, 9544786324, 488174658, 9210076976, 864190149, 8328870937, 1188900157, 8005203607, 765010284, 8638409778, 166529311, 9326591748, 57194082, 8907343957, 660242446, 8270761641, 1177075128, 7941394142, 1196657469, 8196375507, 475786322, 8151556459, 387573310, 7624292526, 778800668, 6924858402, 828150905, 6972011653, 36936460, 7364152427, $-523531189$, 7603045547, $-522042227$, 6942706650, 218982829, 6125518634, 669337729, 6123178426, 569668420, 6591787183, 75768606, 6400974331, 36431967, 5893316966, 101388020, 5319695264, 102495640, 5231238345, $-682275789$, 5286021245, $-1039275885$, 5013380379, $-634603921$, 4303406163, 30415910, 3658451229, 147125795, 3460063055, 75926902, 3703180326, $-323853178$, 3653655136, $-238392752$, 3198652693, $-165979401$, 2750313619, $-104932882$, 2492617006, $-365121077$, 1964774200, $-432485961$, 1418575903, $-346397268$, 795227753, 69588008, 117183978, 17091166, $-84709442$, 30059544].
\end{flushleft}
\normalsize

This element has norm $10753$:

\tiny\noindent
\begin{flushleft}
[2115315, 294536700, $-1259710$, 290751233, $-658734$, 286253190, 6202172, 290637083, 1020656, 290423296, $-785288$, 286809730, 431913, 290849807, $-1972309$, 289141750, 1493136, 277079280, 5027120, 281808363, 2632769, 282772045, 2606799, 275911224, $-1617886$, 279800829, $-2830788$, 277631577, 6603296, 268177615, 1887075, 267705912, 2525056, 263121511, 8861721, 262918374, $-2519628$, 265098198, $-2478906$, 254252249, 7330168, 255187684, 1527574, 249540354, 3108613, 238201953, 6373217, 243789683, $-264411$, 242990286, 2452961, 229857806, 49331, 232535039, 2084912, 223094230, 9146065, 215098451, 115902, 219852521, 387854, 211046359, 6617566, 208066712, $-5512979$, 203853094, 2938804, 191528668, 10507462, 191238633, $-124009$, 186251412, 976595, 177087267, 1202043, 182974650, $-2866736$, 170675998, 2795285, 158104278, 1621262, 160310188, 2372332, 148894811, 5059504, 143612130, $-6568863$, 146665733, $-1471899$, 134710036, 2866574, 126299884, $-3787243$, 120781480, 2780159, 111599813, 3927794, 110874132, $-4748117$, 102000659, $-2104084$, 96925239, $-305319$, 91713010, $-760194$, 77534598, 2343950, 71004081, $-388880$, 72762622, $-429211$, 60416092, $-2906516$, 53531722, $-4333035$, 49933247, 1898531, 37573906, 3511913, 27924006, $-1569879$, 24254418, 65830, 21400232, $-1564070$,12689918, $-4103286$, 4118337].
\end{flushleft}
\normalsize

This element has norm $11777$:

\tiny\noindent
[1309, 111, 323, 687, 443, 1010, 109, 133, 384, 217, 263, 610, 12, 183, $-2$, 663, 446, 1483, 241, 407, $-32$, 848, $-145$, 455, $-982$, 157, $-434$, 1121, $-320$, 789, 671, $-16$, 194, 752, $-233$, 1191, $-367$, 1, 382, 287, 23, 794, 488, 78, 125, $-416$, $-28$, 519, 1231, $-387$, 817, 479, 1294, 736, 697, $-789$, 45, 220, 291, 126, 1277, $-1062$, 577, $-67$, 1028, 1270, 788, $-567$, 719, 46, 716, $-105$, 6, $-16$, 80, 75, $-194$, 72, 657, $-319$, $-311$, 110, 1058, 691, 173, 39, $-329$, 164, $-388$, 241, 524, $-45$, $-381$, $-1016$, $-60$, 845, 110, 109, 508, 30, 1454, $-291$, $-76$, 17, 41, $-465$, 242, $-49$, 854, $-286$, $-57$, $-612$, 1553, $-54$, 364, 11, 948, $-428$, 189, $-416$, 612, 319, 515, $-854$, 237, 204, $-978$, $-27$].
\normalsize

This element has norm $12289$:

\tiny\noindent
\begin{flushleft}
[617693837477, 2370075244431, $-339648780201$, 2238090913237, $-498204865353$, 2221056021584, $-152042916658$, 3110568633642, $-31712501506$, 1411860086389, 80359803004, 2546233401730, $-182447396621$, 2299547040386, $-290214250390$, 2590997629540, $-326219315631$, 2312166834647, 48131827765, 1783289384553, 378442675003, 2329438770514, $-580294589938$, 2218401635376, $-717839827005$, 3197597713419, 704535529089, 1130283380417, $-184355619389$, 2116839801895, $-414187025834$, 2857033836333, $-217842459574$, 1843457953282, 62852900411, 2356758151872, $-228936921555$, 1651643367144, $-266752316105$, 2381610344981, 466861508085, 1879106118549, $-578325063236$, 2244797414571, $-402842736038$, 2139817462475, 171310567795, 1362268990875, 224229141506, 2092865456148, $-301525068860$, 2081510046657, $-398952907693$, 2330201794229, $-255178406315$, 1066318926881, 493751897628, 1729740541006, $-103759157461$, 2173771486057, $-620363038292$, 1867081383268, $-151277221788$, 1353941271448, 179948629945, 1479240060095, 342020052730, 2184516411619, $-717841372895$, 650389362044, $-44740662985$, 2315655696498, $-4478644809$, 1378542962891, $-77525487333$, 893047645591, 179430931459, 1668844385937, $-220938039747$, 1238056324819, $-388417780359$, 1758644487369, 49162113897, 574831985668, 101296936710, 1580353533011, 106296616993, 578837428302, $-424400197895$, 1585976564060, $-288654893227$, 1661210261127, 439833142557, $-247202126371$, $-189885602018$, 1069320100511, 14965483619, 1372382819001, $-350945437549$, 1166656722085, $-26875340031$, $-112046953946$, 264482159358, 971939323123, 315443374, 737076751102, $-242050923233$, 625328766932, $-163049388264$, 651831316391, 103505212143, 150204548047, 48732385133, 676366606419, 22487056291, 50736674216, 12330433962, 876281794400, $-244619123714$, $-23141241475$, $-425071710999$, 173800142051, 871047731302, $-22908844162$, 21071113068, 33711180585, -856494602158, 969230569158].
\end{flushleft}
\normalsize

This element has norm $12799$:

\tiny\noindent
\begin{flushleft}
[184037827075, 27497694581842, $-19626196308$, 27547128880119, 110431733149, 27555147977579, $-43762006334$, 27628048533330, 363210681747, 27670842645800, 33491843358, 27407323326486, 288418234223, 27749482284368, 315874648397, 27372865864009, 330043125975, 27299903871399, 340523125682, 27015368233356, 391243147017, 27071327899417, 288329303744, 26642786767103, 149912398940, 26500392938787, 109894359842, 26144374880238, 146407833367, 26161831207103, $-57910866992$, 25492292196379, $-274642305172$, 25540457641474, 142369219973, 25160228118999, $-129682164625$, 24702675340198, $-69429495725$, 24538098116712, 46349415635, 24444616969505, 156136351425, 23931205738586, $-2961224626$, 23705831310594, 256615446179, 23524258730874, 140896966972, 23256896414578, 205005339494, 22739364944777, $-15843992634$, 22345779576504, 357116818963, 22065388666012, 80818167553, 21316751144174, $-43890778334$, 20910703796699, 42084499799, 20484469377007, 77475653685, 19785027507686, $-263756725875$, 19271681904259, $-103004425019$, 19020439605835, $-269420692984$, 18334556916614, $-208404770122$, 17932081887237, $-248797915350$, 17310365165180, $-126861808332$, 17069655283837, $-89882204515$, 16366181711576, $-70223384746$, 15894548371886, 59088655957, 15314097572833, 224287929661, 14934547282661, 44599857235, 14137788660464, 57486988045, 13843047543817, 144867254487, 12982325354549, $-13428448847$, 12509458887132, 8304629546, 11689066750504, $-249842160599$, 11153328440462, $-30856843432$, 10351054047936, $-226948948743$, 9587784962712, $-178179267698$, 8855105075807, $-263989946416$, 8459099826011, $-46611349150$, 7581871533541, $-417612144252$, 6960647418598, 53356750842, 6559779943246, $-164811255782$, 5847169871093, $-22801282585$, 5276876192677, $-133695961049$, 4561501186914, 267900193832, 4087343983235, 106671432015, 3102753616610, 112993348875, 2619490858980, 125041798071, 1829389680073, 297752663615, 1181872802028, -53529707223, 69559134213].
\end{flushleft}
\normalsize

This element has norm $13313$:

\tiny\noindent
\begin{flushleft}
[8916659723289, 47268532674, 8908556110733, 70481864656, 8937291165906, 67739426562, 8861569846044, $-7725649614$, 8829707041453, 47638278771, 8861390871712, 48133856326, 8764144540244, $-48910258200$, 8687350180924, 1357021369, 8733804704947, 29074824782, 8616774037745, $-70624616649$, 8540760686368, $-24859730341$, 8532405397014, $-11889210622$, 8459821763982, $-28518190023$, 8372357336801, $-56800374255$, 8301117283526, $-11761801095$, 8281446215115, 17501367664, 8184340584635, $-33759340894$, 8063590959098, $-5354668175$, 8064978655183, 85625679864, 7966412060993, $-8044159715$, 7800415831985, 33108335988, 7800989809347, 100637609567, 7667310373669, 32871123796, 7512849402161, 35302340094, 7428424479889, 77578363916, 7318557559367, 52248738806, 7146998380511, 14965239152, 7002340196323, 20703950879, 6889137463326, 58623632102, 6752305608406, $-14386714403$, 6508893132068, $-45339801496$, 6432979746206, 50051731390, 6282439291691, $-47946996920$, 6027018399904, $-74116504992$, 5922613371875, 7112222086, 5791260006111, $-33927849059$, 5550734846917, $-74105570364$, 5414865621403, $-7413305441$, 5276451123686, $-5031097618$, 5096944465149, $-12408772293$, 4898309553471, $-26654004178$, 4753486648791, 61879736346, 4623095676929, 27625581497, 4356719254730, $-11365677549$, 4213264031932, 85683966203, 4081483059666, 64666311773, 3800571512362, $-11699782794$, 3619743124051, 84748679823, 3489223839999, 61110804212, 3204599108570, 3360620284, 3005433193669, 28675626079, 2812542786776, 40010051208, 2606574214541, 6075789802, 2332601110596, $-49927165145$, 2123416470081, 2846966847, 1957178717014, $-2476742731$, 1677693344486, $-96508928231$, 1427324401528, $-31753778165$, 1316756845314, 5835444314, 1028165661555, $-97373171807$, 795695608823, $-24168787111$, 656098328404, 2267514151, 418089679287, $-30049678195$, 184507943986, $-25303639014$].
\end{flushleft}
\normalsize

\end{proof}

\section{Finding principal ideals of small norm in $\Q(\zeta_{512}+\zeta_{512}^{-1})$}
While it is a straightforward matter to verify that the above elements have the desired norms, actually finding these elements poses a challenge.  Since we must search a lattice of 128 dimensions, a brute force approach of searching over a suitably sized ``box" is impractical.  For example, given an integral basis, if we were to search all elements with coefficients between $-2$ and $2$, that would mean checking the norms of $5^{128} \approx 10^{89}$ elements, which is substantially larger than the number of particles in the universe!

A more practical approach is to search over ``sparse" vectors, i.e. vectors where almost all the coefficients are zero.  The hope would be that we could find the desired elements of small prime norm, or elements that factor over primes of small norm and produce relations in the class group.

First we'll describe this process for the smaller field, $\Q(\zeta_{256}+\zeta_{256}^{-1})$, and then contrast it with the situation presented by $\Q(\zeta_{512}+\zeta_{512}^{-1})$.

\begin{exa}[Finding elements of small prime norm in $\Q(\zeta_{256}+\zeta_{256}^{-1})$]\label{Example256}
The goal is to find algebraic integers in $\Q(\zeta_{256}+\zeta_{256}^{-1})$ which have norms that are prime and congruent to $\pm 1$ modulo 256, i.e. the primes which totally split.  The ten smallest of these are 257, 769, 1279, 3329, 3583, 5119, 6143, 6911, 7681, and 7937.

The integral basis that we'll use is $\{b_0, b_1,\dots, b_{63}\}$ where $b_0 = 1$ and $b_j = 2 \cos{(2 \pi j/256)}$ for $j$ from $1$ to $63$.  We will search over sparse vectors where at most six coefficients are nonzero, and the nonzero coefficients are either $1$ or $-1$.  When we search over these sparse vectors, we do indeed find the ten prime norms we were looking for.    For example, the element $b_0 + b_1 +  b_{14} $ has norm 257, and the element $b_0 - b_3 + b_4 - b_{22} - b_{34} - b_{53}$ has norm 6143.
\end{exa}

What happens if we repeat this above process with similar sparse vectors for the larger field $\Q(\zeta_{512}+\zeta_{512}^{-1})$?  Unfortunately, we do not find any elements of small prime norm.  In fact, the two smallest prime norms found this way are rather large: 6147073 and 9627649.  In this respect, the properties of the field $\Q(\zeta_{512}+\zeta_{512}^{-1})$ are markedly different than that of the smaller field $\Q(\zeta_{256}+\zeta_{256}^{-1})$, so another approach is required.

Perhaps the best way to illustrate the approach used is to explicitly write down the particular calculations.  We start with the integral basis $\{b_0, b_1,\dots, b_{127}\}$ where $b_0 = 1$ and $b_j = 2 \cos{(2 \pi j/512)}$ for $j$ from $1$ to $127$.  As mentioned above, searching over sparse vectors led to two elements of prime norm:
\[N(b_0 + b_1 - b_8 - b_9 + b_{48}) = 6147073\]
\[N(b_0 + b_1 +b_2  + b_9 + b_{10} - b_{48}) = 9627649\]

We may also consider algebraic integers with norms which are even, since we can always repeatedly divide by any element of norm $2$, such as $b_1$, until we get an algebraic integer of odd norm.  For example,
\[b_1^{-2}(b_1 + b_{10} + b_{26} +b_{35} +b_{43} +b_{50} +b_{52} +b_{95}),\]
is algebraic integer of norm $1142783$.

Another useful element is
\[b_1^{-1}(b_1 + b_{18} - b_{39} + b_{108} + b_{127}),\]
which has norm $9627649 \cdot 2078207$.  This produces a relationship in the class group between a prime of norm $9627649$, which is known to be principal, and a prime of norm $2078207$.  Therefore all primes of norm $2078207$ are principal, and there exists a unique Galois automorphism $\sigma$ for which
\[\frac{b_1^{-1}(b_1 + b_{18} - b_{39} + b_{108} + b_{127})}{\sigma(b_0 + b_1 +b_2  + b_9 + b_{10} - b_{48})}\]
is an algebraic integer of norm $2078207$. 

However, further search by the author of sparse vectors using the basis $\{b_0,\dots,b_{127}\}$ found neither elements of small prime norm nor elements which produce useful relations in the class group.  To find more suitable elements, we choose a different basis over which to search sparse vectors.  For $k$ from $0$ to $127$, put
\[c_k = \sum_{j=0}^k b_j.\]
The $c_k$ are the cyclotomic units, and $\{c_0,\dots,c_{127}\}$ form an integral basis.  Sparse vectors over this basis can produce some algebraic integers of relatively small norm or produce interesting class group relations.  Two algebraic integers that prove to be of critical importance are
\[ c_0 + c_6 + c_{15} + c_{39} - c_{104} + c_{111} + c_{120}\]
which has norm $2078207 \cdot 6215646209$, and
\[c_0 + c_3 + c_{19} + c_{64} - c_{71} + c_{103} + c_{119}\]
which has norm $6143^2 \cdot 6215646209$.
By choosing the appropriate Galois conjugates and taking quotients, we can explicitly write down an algebraic integer of norm $6215646209$ and then also an algebraic integer $\alpha$ of norm $6143^2$.

In the subfield $\Q(\zeta_{256}+\zeta_{256}^{-1})$, it is easy to find an algebraic integer of norm $6143$ and by including that element in the larger field, we produce an algebraic integer $\beta$ of norm $6143^2$.  One example is
\[\beta = b_0 - b_6 + b_8 - b_{44} - b_{68} - b_{106}.\]

Now given $\alpha$ and $\beta$ as above, which both have norm $6143^2$, consider the quotients $\sigma(\alpha)/\beta$ for each Galois automorphism $\sigma$.  A calculation shows that none of the quotients are algebraic integers.  Thus, we have an inequality of principal ideals
\[(\sigma(\alpha)) \neq (\beta)\]
for every $\sigma$.

Let $\eta$ denote a generator of the Galois group $\op{Gal}(\Q(\zeta_{512}+\zeta_{512}^{-1})/\Q)$, which is cyclic of order 128.  Then $\eta^{64}$ fixes the subfield $\Q(\zeta_{256}+\zeta_{256}^{-1})$, and
\[ (\beta) = P \cdot \eta^{64}(P)\]
where $P$ is some prime ideal of norm $6143$ (noting that, as a prime of $\Q(\zeta_{256}+\zeta_{256}^{-1})$, $\beta$ lies over 6143 so it does indeed split in $\Q(\zeta_{512}+\zeta_{512}^{-1})$).

For suitably chosen automorphism $\sigma$, we have that
\[ (\sigma(\alpha)) = \tau(P) \cdot \eta^{64}(P)\]
where $\tau$ is not the identity automorphism.  Taking quotients shows
\[\frac{P}{\tau(P)}\]
is a principal fractional ideal.  Suppose $\tau$ has order $m$ in the Galois group.  Since $\tau$ is not the identity automorphism, $m$ must be even, so
\[\frac{P}{\eta^{64}(P)} = \frac{P}{\tau^{m/2}(P)} = \frac{P}{\tau(P)} \frac{\tau(P)}{\tau^2(P)} \cdots \frac{\tau^{m/2 - 1}P}{\tau^{m/2}(P)}\]
is a principal fractional ideal.  Thus
\[P^2 = (\beta) \frac{P}{\eta^{64}(P)}\]
is a principal ideal.  But Weber \cite{Weber} showed that the class number is odd, so $P$ itself must be a principal ideal of norm 6143.  This approach can be further extended to calculate an actual generator of $P$.

Once we have shown that a prime of such small norm is principal, it is relatively easy to use sparse vectors to generate more class group relations to show that other prime ideals of small norm are principal and to find their generators.

\section{Concluding remarks}
The technique described in this paper for calculating class numbers is applicable to other totally real fields of large discriminant, provided that we can establish a class number upper bound by counting sufficiently many prime ideals of small norm in the Hilbert class field.  This allows us to attack the class number problem for a large number of totally real fields which have not been treatable by previous methods.  By extension, we can address the class number problem for CM fields, provided we have information about the relative class number.

\subsection*{Acknowledgements}
I would like to thank my advisor, Henryk Iwaniec, for introducing me to Weber's class number problem and his steadfast encouragement.  I am exceedingly thankful for the careful reading and suggestions of Jerrold Tunnell and Lawrence Washington.  I appreciate the interest and feedback of Stephen D. Miller, Ren\'e Schoof and Christopher Skinner.  Finally, I would like to thank Takashi Fukuda, Keiichi Komatsu and Takayuki Morisawa for our fruitful discussions about the Weber problem and much else.

\end{document}